\documentclass[12pt]{amsart}

\usepackage{fullpage,mathrsfs,amssymb,color,url,amsmath,mathtools,enumerate,}
\usepackage{tikz-cd}
\usepackage[all,cmtip]{xy}
\definecolor{hot}{RGB}{65,105,225}
\usepackage[pagebackref=true,colorlinks=true, linkcolor=hot ,  citecolor=hot, urlcolor=hot]{hyperref}

\def\sD{\mathscr{D}}
\def\ra{\rightarrow}
\def\bk{\mathbf{k}}
\def\bs{\mathbf{s}}
\def\bf{\mathbf{f}}
\def\bi{\mathbf{i}}
\def\ba{\mathbf{a}}

\def\bbe{\mathbf{e}}
\def\bal{{\boldsymbol{\alpha}}}

\def\bla{{\boldsymbol{\lambda}}}
\def\bC{\mathbb{C}}
\def\bK{\mathbb{K}}
\def\bQ{\mathbb{Q}}
\def\bD{\mathbb{D}}
\def\bZ{\mathbb{Z}}
\def\bN{\mathbb{N}}
\def\cS{\mathcal{S}}
\def\cL{\mathcal{L}}
\def\cP{\mathcal{P}}
\def\cH{\mathcal{H}}
\def\cM{\mathcal{M}}
\def\cN{\mathcal{N}}
\def\cQ{\mathcal{Q}}
\def\al{\alpha}
\def\sA{\mathscr{A}}
\def\sO{\mathscr{O}}

\def\sN{\mathscr{N}}
\def\sX{\mathscr{X}}

\newcommand{\Ann}{\textup{Ann}}
\newcommand{\supp}{\textup{supp}}

\newcommand{\gr}{\textup{gr}}

\newcommand{\Spec}{\textup{Spec}\,}
\newcommand{\D}{\mathbb{D}}
\newcommand{\DR}{\textup{DR}}
\newcommand{\Ch}{\textup{Ch}}
\newcommand{\Chr}{\textup{Ch}^{\textup{rel}}}
\newcommand{\be}{\begin{equation} }
\newcommand{\ee}{\end{equation} }

\def\Exp{\textup{Exp}}
\def\CC{\textup{CC}}

\def\chr{\textup{Ch}^{rel}}
\def\Ext{\textup{Ext}}
\def\Rhom{\textup{RHom}}
\newcommand{\xto}{\xrightarrow}
\def\CCr{\CC^{\textup{rel}}}

\theoremstyle{plain}
\newtheorem{theorem}{Theorem}[subsection]
\newtheorem{lemma}[theorem]{Lemma}
\newtheorem{prop}[theorem]{Proposition}
\newtheorem{coro}[theorem]{Corollary}

\theoremstyle{definition}
\newtheorem{rmk}[theorem]{Remark}

\newtheorem{definition}[theorem]{Definition}

\title{Zero loci of Bernstein-Sato ideals - II}
\author{Nero Budur}
\address{KU Leuven, Celestijnenlaan 200B, B-3001 Leuven, Belgium} 
\email{nero.budur@kuleuven.be}
\author{Robin van der Veer}
\address{KU Leuven, Celestijnenlaan 200B, B-3001 Leuven, Belgium} 
\email{robin.vanderveer@kuleuven.be}
\author{Lei Wu}
\address{Department of Mathematics, University of Utah, 155 S. 14000 E, Salt Lake City, UT 84112, USA}
\email{lwu@math.utah.edu}
\author{Peng Zhou}
\address{Department of Mathematics, University of California at Berkeley, 931 Evans Hall, Berkeley, CA 94720, USA}
\email{pzhou.math@gmail.com}


\numberwithin{equation}{section}

\begin{document}

\begin{abstract} We have recently proved a precise relation between Bernstein-Sato ideals of collections of polynomials and monodromy of generalized nearby cycles. In this article we extend this result to other ideals of Bernstein-Sato type. 
\end{abstract}

\maketitle

\tableofcontents

\section{Introduction}

\subsection{Analytic invariants.}\label{subsUL} 
Let $$F=(f_1,\ldots,f_r):X\ra\bC^r$$ be a morphism of smooth complex affine irreducible algebraic varieties, or the germ at $x\in X$ of a holomorphic map on a complex manifold. Let $$\ba=(a_1,\ldots, a_r)\in\bN^r.$$ One defines an ideal of Bernstein-Sato type
$$
B_F^{\,\ba}=\{ b\in\bC[s_1,\ldots,s_r]\mid b\prod_{i=1}^rf_i^{s_i}=P\cdot\prod_{i=1}^rf_i^{s_i+a_i}\text{ for some }P\in\sD_X[s_1,\ldots,s_r]\},
$$
where $\sD_X$ is the ring of linear differential operators on $X$ and $s_i$ are independent variables. The zero locus of this ideal is denoted $$Z(B_F^{\,\ba})\subset\bC^r.$$ In this article we address the structure and the precise topological information contained by $Z(B_F^{\,\ba})$. Our first main result is:

\begin{theorem}\label{thrmMoreA} Let $F=(f_1,\ldots,f_r):X\ra\bC^r$ be a morphism of smooth complex affine irreducible algebraic varieties, or the germ at $x\in X$ of a holomorphic map on a complex manifold.  Let $\ba\in\bN^r$ such that $\bf^\ba=\prod_{i=1}^rf_i^{a_i}$ is not invertible. Then: 
\begin{enumerate}[(i)]
\item Every irreducible component of $Z(B_F^{\,\ba})$ of codimension 1 is a  hyperplane of type $l_1s_1+\ldots+l_rs_r+b=0$ with $l_i\in\bQ_{\ge 0}$, $b\in\bQ_{>0}$, and for each such hyperplane there exists $i$ with $a_i\ne 0$ such that $l_i>0$. 
\item Every irreducible component of  $Z(B_F^{\,\ba})$ of codimension $>1$  can be translated by an element of $\bZ^r$ inside a component of codimension 1.
\end{enumerate}
\end{theorem}

The first claim without the strict positivity of $l_i$ is due to Sabbah \cite{Sab} and Gyoja \cite{G}. The second claim was proven by Maisonobe \cite{M} for the usual Bernstein-Sato ideal 
$$
B_F=B_F^{\mathbf{1}}
$$
with $\mathbf{1}=(1,\ldots,1)$. 

The proof of Theorem \ref{thrmMoreA} is obtained by extending arguments from \cite{M} and \cite{BVWZ}.

The connection with topology is via the exponential map $$\textup{Exp}:\bC^r\ra (\bC^*)^r,\quad (\al_1,\ldots,\al_r)\mapsto (\exp(2\pi i{\al_1}),\ldots, \exp(2\pi i{\al_r})).$$
For $i\in\{1,\ldots, r\}$ we define 
$$
B_{F,i} = B_F^{\,\bbe_i}
$$
with  $\bbe_i=(0,\ldots,0,1,0,\ldots,0)$  the $i$-th standard basis vector. If $f_i$ is invertible, $B_{F,i}=(1)$ and $Z(B_{F,i})=\emptyset$. We recall here that
\be\label{eqIII}
\textup{Exp}(Z(B_F^{\,\ba}))=\bigcup_{i}\textup{Exp} (Z(B_{F,i})) 
\ee
where the union is taken over all $1\le i\le r$ with $a_i\ne 0$ and $f_i$ not invertible, by \cite[Lemma 4.17]{B-ls}.

\subsection{Topological invariants.}
Let $\psi_F\bC_X$ be Sabbah's specialization complex, see \cite[\S 2]{BVWZ} for definition. This  is a generalization of Deligne's nearby cycles complex, the monodromy action being replaced by $r$ simultaneous monodromy actions, one for each function $f_i$. Let $$\cS(F) \subset (\bC^*)^r$$  be the support of this monodromy action on $\psi_F\bC_X$. When $r=1$, this is the set of eigenvalues of the monodromy on the nearby cycles complex. The support $\cS(F)$ has a few other topological interpretations, one being in terms of cohomology support loci of rank one local systems, see \cite[\S 2]{BVWZ}.

If $f_i$ is not invertible, we let 
$$
\cS_i(F) = \textup{Supp}((\psi_F\bC_X)|_{ f^{-1}_i(0)}) \subset (\bC^*)^r
$$
be the support of the monodromy action on the restriction of $\psi_F\bC_X$ to the zero locus of $f_i$.

More generally, if $\bf^\ba$ is not invertible, we let
$$
\cS_\ba(F)  \subset (\bC^*)^r
$$
be the support of the monodromy action on the restriction of $\psi_F\bC_X$ to the zero locus of $\bf^\ba$, so that \be\label{eqI}
\cS_\ba(F)= \bigcup_{i}\cS_i(F)
\ee 
where the union is taken over all $1\le i\le r$ with $a_i\ne 0$ and $f_i$ not invertible. With the same assumptions as in Theorem \ref{thrmMoreA}, we have:

\begin{theorem}\label{thrmMoreB}
$\cS_\ba(F)$ underlies a scheme defined over $\bQ$ and is a finite union of torsion-translated complex affine algebraic subtori of codimension 1 in $(\bC^*)^r$.
\end{theorem}

The case $\ba=\mathbf{1}$ is due to \cite{BW16, BLSW}. We give two proofs of this result, one of them by adapting the argument from \cite{BLSW}. A second proof follows directly from Theorem \ref{thrmMoreA} together with the next theorem.

\subsection{The connection.}  

\begin{theorem}\label{thrmMoreC}  With the same assumptions as in Theorem \ref{thrmMoreA} we have $$\textup{Exp} (Z(B_F^{\,\ba})) = \cS_\ba(F).$$ 
\end{theorem}

In particular, $\textup{Exp} (Z(B_{F,i})) = \cS_i(F)$ if $f_i$ is not invertible.

This refines the corresponding statement for the ideal $B_F$, in which case the inclusion of the topological side in the algebraic side was shown in \cite{B-ls}, and the reverse inclusion was finally shown in \cite{BVWZ}. The proof is by extension of arguments from \cite{B-ls, WZ, BVWZ}.

\subsection{}  In Section \ref{secExR} we give the first proof of Theorem  \ref{thrmMoreB} and give a $\sD$-module  interpretation of the restricted support loci. This relies on an explicit case of the Riemann-Hilbert correspondence, Proposition \ref{lemRH2}, whose proof we postpone to the last section, Section \ref{secRH2}. In Section \ref{secRHM} we show how Theorems \ref{thrmMoreA} and \ref{thrmMoreC} follow from a technical result on $\sD$-modules, Theorem \ref{thmMais}. Section \ref{secSl} is the core of the article and is devoted to the proof of Theorem \ref{thmMais}.

\subsection{Acknowledgement.} The first author was partly supported by the grants STRT/13/005 and Methusalem METH/15/026 from KU Leuven, G097819N and G0F4216N from the Research Foundation - Flanders. The second author is supported by a PhD Fellowship of the Research Foundation - Flanders. The fourth author is supported by the Simons Postdoctoral Fellowship as part of the Simons Collaboration on HMS.

\section{The restricted support loci}\label{secExR} 

In \cite[\S 2]{BVWZ} we gave various interpretations of the support locus $\cS(F)$. In this section we refine these descriptions to address $\cS_i(F)$ and we prove Theorem \ref{thrmMoreB}.

\subsection{Notation}\label{subDTE} Throughout the article we use the notation and definitions from \cite[\S2 ]{BVWZ}. We let $F=(f_1,\ldots, f_r):X\ra\bC^r$ be as in Theorem \ref{thrmMoreA}. We let $n=\dim X$, $f=\prod_{i=1}^rf_i$, $D=f^{-1}(0)$, $D_i=f_i^{-1}(0)$, $U=X\setminus D$, $i:D\ra X$ is the closed embedding, and $j:U\ra X$ is the open embedding. 

We  let $\bs=(s_1,\ldots,s_r)$, $\bf^\bs=\prod_{i=1}^rf_i^{s_i}$, and in general tuples of numbers will be in bold, e.g. $\mathbf{1}=(1,\ldots, 1)$, $\bal=(\al_1,\ldots,\al_r)$, etc. 

We denote by $D^b_{c}(\bC_X)$ the bounded derived category of constructible sheaves on $X$, and by $Perv(X)$ the category of perverse sheaves on $X$.

If $\bla\in(\bC^*)^r$, we let $L_\bla$ be the rank one $\bC$-local system on $U$ obtained as the pullback via $F:U\ra(\bC^*)^r$ of the rank one local system on $(\bC^*)^r$ with monodromy $\lambda_i$
 around the $i$-th missing coordinate hyperplane. 

We fix for the rest of the section $i$ in $\{1,\ldots, r\}$ such that $f_i$ is not invertible. 

\subsection{Non-simple extension loci} We give a refinement in terms of $\cS_i(F)$ of the following description of $\cS(F)$ as a locus of rank one local systems on $U$ with non-simple higher direct image to $X$ from \cite[\S 1.4]{BLSW}:

\begin{prop}
$$
\cS(F)=\{\bla \in(\bC^*)^r\mid \textup{Cone}(j_{!}L_\bla[n] \ra Rj_*L_\bla[n])\ne 0 \text{ in }D^b_c(\bC_X)\}.
$$
\end{prop}

From the description of Sabbah's complex $\psi_F\bC_X$ from \cite{BLSW}, one has  the following equivalent definition for $\cS_i(F)$:
\begin{lemma}
$$
\cS_i(F)=\{\bla \in(\bC^*)^r\mid \textup{Cone}(j_{!}L_\bla[n] \ra Rj_*L_\bla[n])_{| D_i}\ne 0 \text{ in }D^b_c(\bC_{D_i})\}.
$$
\end{lemma}

A more useful description for us will be the following. Let $g$ and $h$ be the open embeddings 
\be\label{eqGH}
U\xrightarrow{g} X\setminus {D_i}\xrightarrow{h} X
\ee
so that $j=h\circ g$. The result we are after in this subsection is the following:

\begin{prop}\label{propNSe}
We have
$$
\cS_i(F)=\left\{\bla\in(\bC^*)^r \mid \frac{Rj_*L_\bla[n]}{h_{!*}Rg_*L_\bla[n]}\ne 0\text{ in }Perv(X)\right\},
$$
or equivalently,
$$
\cS_i(F)=\left\{\bla\in(\bC^*)^r \mid \textup{Cone}({h_{!}Rg_*L_\bla[n]}\ra{Rj_*L_\bla[n]})\ne 0\text{ in }D^b_c(\bC_X)\right\}.
$$
\end{prop}
\begin{proof} Set $V=X\setminus {D_i}$. One has in $D^b_c(\bC_X)$ two distinguished triangles, namely the adjunction triangles corresponding to $U$ and $V$,
\be\label{eqFES}
j_{!}L_\bla[n] \ra Rj_*L_\bla[n]\ra i_*i^{-1}Rj_*L_\bla[n]\mathop{\longrightarrow}^{[1]} 
\ee
and
\be\label{eqSES}
h_{!} Rg_* L_\bla[n] \ra Rh_{*} Rg_* L_\bla[n] \ra \iota_{*}\iota^{-1}Rh_*Rg_* L_\bla[n]\mathop{\longrightarrow}^{[1]},
\ee
where $i:D\ra X$ and $\iota:D_i\ra X$ are the closed embeddings. In both cases, the first two terms are perverse, and the images as perverse sheaves of the left-most maps in the two complexes are the intermediate extensions. 

 By taking the long exact sequence of perverse cohomology of (\ref{eqSES}), we have an exact sequence in $Perv(X)$
$$
0\ra K\ra h_{!} Rg_* L_\bla[n] \ra Rj_* L_\bla[n] \ra C\ra 0
$$
with
$$
K={}^pH^{-1}(\iota_*\iota^{-1}Rj_*L_\bla[n])\quad\text{ and }\quad C= {}^pH^{0}(\iota_*\iota^{-1}Rj_*L_\bla[n]),
$$ 
where ${}^pH^*$ are the perverse cohomology sheaves. On the other hand, $K$ and $C$ have an alternative description, namely they fit into an exact sequence of perverse sheaves
$$
0\ra K \ra {}^p\psi_{f_i}Rg_* L_\bla[n] \xrightarrow{T-id} {}^p\psi_{f_i}Rg_* L_\bla[n]\ra C\ra 0,
$$
where ${}^p\psi_{f_i}=\psi_{f_i}[-1]$ is the perverse nearby cycles functor of $f_i$ and $T$ is the monodromy action around $D_i$, see \cite[5.8]{DM}. In particular, $K$ and $C$ have  same length as perverse sheaves, and so they vanish or not simultaneously. Thus $$K=0 \iff C=0\iff \iota^{-1}Rj_* L_\bla[n]=0$$ in the derived category. 
This is further equivalent to
$$
\textup{Cone}(j_{!}L_\bla[n] \ra Rj_*L_\bla[n])_{\mid D_i}=0.
$$
by (\ref{eqFES}), since
$$
\iota^{-1}Rj_* L_\bla[n] = \iota^{-1}(i_*i^{-1}Rj_*L_\bla[n]).
$$
Hence the proposition follows from the previous lemma.
\end{proof}

\subsection{Proof of Theorem \ref{thrmMoreB}.}   
By (\ref{eqI}), it is enough to consider $\cS_i(F)$.

The description of $\cS_i(F)$ from Proposition \ref{propNSe} allows one to apply the general results of \cite{BW16} to conclude that $\cS_i(F)$ is the set of complex points of a $\bQ$-scheme, and it is a finite union of torsion-translated complex affine algebraic subtori of $(\bC^*)^r$. 

It remains to prove that each irreducible component of $\cS_i(F)$ has codimension one. It was proved in \cite[Theorem 1.3]{BLSW} that $\cS(F)$ satisfies this property by showing that it is the union over points $x$ in $D$ of the hyperplanes appearing as zero or polar loci of the monodromy zeta function $Z_{F,x}^{mon}(t_1,\ldots,t_r)$ of $F$ at $x$. By definition $\cS(F)$ contains this locus. For the other inclusion, it is enough to take a torsion point in the support of $(\psi_F\bC_X)_x$ and show that it lies on a hyperplane in the zero or polar locus of $Z_{F,y}^{mon}(t_1,\ldots,t_r)$ for some $y$ in $D$. The proof is by reduction to the case $r=1$. The reduction step uses the comparison due to Sabbah  $$Z_{F,x}^{mon}(t^{m_1},\ldots,t^{m_r})=Z_{g,x}(t)$$ for integers $m_i>0$ carefully chosen in terms of the torsion point, and $g=f_1^{m_1}\ldots f_r^{m_r}$. The $r=1$ case is a result of Denef stating that an eigenvalue of the monodromy on $(\psi_g\bC_X)_x$ always appears as zero or pole of $Z_{g,y}^{mon}(t)$ for some $y\in D$ close to $x$, this result depending only on the perversity of  $\psi_g\bC_X[n-1]$.

The proof from {\it loc. cit.} thus extends   word-by-word to prove our claim if one replaces $\psi_F\bC_X$ by $(\psi_F\bC_X)|_{D_i}$,  $\psi_g\bC_X$ by $(\psi_g\bC_X)|_{D_i}$, since in this case Denef's theorem for the perverse sheaf $(\psi_g\bC_X)|_{D_i}[n-1]$ gives that one can find such $y$ in $D_i$. $\hfill\Box$

\begin{rmk} A completely different proof of Theorem \ref{thrmMoreB} will follow directly from Theorems \ref{thrmMoreA} and \ref{thrmMoreC}.
\end{rmk}

\subsection{$\sD$-module theoretic interpretation}

We have the following complement to \cite[Theorem 2.5.1]{BVWZ} whose proof we will postpone to the last section: 

\begin{prop}\label{lemRH2} Let $F=(f_1,\ldots, f_r):X\ra\bC^r$ be a morphism from a smooth complex algebraic variety of dimension $n$.  Let $\bal\in\bC^r$ and $\bla=\exp(-2\pi i\bal)$. Let $L_\bla$ be the rank one local system on $U$ defined as in \ref{subDTE}, and let $\cM_\bla=L_\bla\otimes_\bC\sO_U$ the corresponding flat line bundle, so that 
$$
\DR_U(\cM_\bla)=L_\bla[n]
$$
as perverse sheaves on $U$. Let $i$ be such that $f_i$ is not invertible. There exists $k_\bal\in\bZ$ depending on $ \bal$ such that for all integers $k, l$ with
 $k_\bal\le k\ll l$ and $\bk=(k,\ldots,k)\in\bZ^r$, there is a natural quasi-isomorphism in $D^b_{rh}(\sD_X)$
$$
\sD_X[\bs]\bf^{\bs-\bk+l\cdot\bbe_i}\otimes_{\bC[\bs]}\bC_\bal=h_!g_*\cM_\bla,
$$
where $\bC_\bal$ is the residue field of $\bal$ in $\bC^r$,  and $g$ and $h$ are as in (\ref{eqGH}).
\end{prop}

This gives a $\sD$-module theoretic interpretation of $\cS_i(F)$. 

\begin{prop}\label{propSEr}
With $F$ and $k_\bal$ as in Proposition \ref{lemRH2},
$$
\cS_i(F) = \Exp \left\{ \bal\in\bC^r\mid \frac{\sD_X[\bs]\bf^{\bs-\bk}}{\sD_X[\bs]\bf^{\bs-\bk+l\bbe_i}}\otimes_{\bC[\bs]}\bC_\bal\ne 0 \text{ for some }k\ge k_\bal\text{ and for all }
l\gg k \right\}
$$
\end{prop}
\begin{proof}
The proof is the same as the proof of \cite[Proposition 2.5.2]{BVWZ}, after replacing $\sD_X[\bs]\bf^{\bs+\bk}$ with $\sD_X[\bs]\bf^{\bs-\bk+l\bbe_i}$, for which we use now the more refined Proposition \ref{propNSe} and Proposition \ref{lemRH2}.
\end{proof}

The last proposition holds in the local analytic case as well, cf. \cite[Remark 2.5.3]{BVWZ}.

\section{Relative holonomic modules}\label{secRHM}

In this section we recall  some results on relative holonomic modules from \cite[\S3]{BVWZ}. Then we prove the main results of the article, up to a technical result which will be the focus of the next section. 

\subsection{} Let $X$ be a smooth complex affine irreducible algebraic variety of dimension $n$. By $\sD_X$ we denote the ring of linear algebraic differential operators on $X$. For  a regular commutative $\bC$-algebra integral domain $R$, we write $$\sA_R=\sD_X\otimes_\bC R\quad\text{and}\quad\sA=\sA_{\bC[\bs]}=\sD_X[\bs].$$

The order filtration from $\sD_X$ extends $R$-linearly to filtration of $\sA_R$, called the {\it relative filtration}. For a left (or right) $\sA_R$-module $N$, we can then talk about good filtrations and of the induced {\it relative characteristic variety} $\chr(N)$, the support of $\gr\, N$ in $T^*X\times\Spec R$. 

For a finitely generated left $\sA_R$ module $N$, one defines the dual in the derived category  of left $\sA_R$-modules by
\[\bD(N)\coloneqq \Rhom_{\sA_R}(N,\sA_R)\otimes_{\sO_X} \omega^{-1}_X[n],\]
where $\omega_X$ is the dualizing module of $X$, and the twist by $\omega^{-1}_X$ is needed only to pass from right $\sA_R$-modules to left ones. If $N$ is a finitely generated right $\sA_R$-module, then 
\[\bD(N)\coloneqq \Rhom_{\sA_R}(N,\sA_R)\otimes \omega_X[n]\]
is a complex of right $\sA_R$-modules.

\begin{definition}
A finitely generated $\sA_R$-module $N$ is {\it relative holonomic over $R$} if 
$$
\chr(N) = \bigcup_w \Lambda_w\times S_w
$$
for some irreducible conic Lagrangian subvarieties $\Lambda_w$ in $T^*X$, and some irreducible closed subvarieties $S_w$ of $\Spec R$. 
\end{definition}

The notion of relative holonomicity seems to have considered firstly by Sabbah, see for example \cite[II, Th\'eor\`eme 3.2]{Sab}.

Note that if $R$ is a field extension of $\bC$, then $N$ being relative holonomic over $R$ is equivalent to $N$ being holonomic in the usual sense over $\sD_{X_R}=\sA_R$, where $X_R=X\times_\bC R$, and therefore equivalent to  $H^i(\bD(N))=0$ if $i\not=0$. If $R$ is not a field, then in general, relative holonomicity does not imply that the derived $\sA_R$-dual has only one cohomology sheaf. This leads one to the following definition cf. \cite[3.3]{BVWZ}: 

\begin{definition} 
A non-zero finitely generated $\sA_R$-module $N$ is {\it $j$-Cohen-Macaulay} if $\Ext^k_{\sA_R}(N,\sA_R)= 0$ if $k\ne j$.  
\end{definition}

Recall the following terminology from \cite[A. IV]{Bj}. The main properties we need are summarized in \cite[\S 4. Appendix]{BVWZ}.

\begin{definition} For a non-zero finitely generated $\sA_R$-module $N$, the {\it grade number} of $N$ is
$$
j(N) =\min \{k\mid \textup{Ext}^k_{\sA_R}(M,\sA_R)\ne 0\}.
$$
The module $N$ is {\it pure}, or {\it $k$-pure}, if $j(N)=j(N')=k$ for every non-zero submodule $N'$.
\end{definition}

For a finitely generated $\sA_R$-module $N$, we write 
\[B_{ N}=\Ann_R( N)\]
and denote by $$Z(B_N)\subset \Spec R$$ 
the reduced subvariety  defined by the radical ideal of $B_{ N}$.

The Cohen-Macaulay property holds at least generically in the following situation, cf. \cite[Lemma 3.5.2]{BVWZ}:

\begin{prop}\label{propGeN}
Let $N$ be a finitely generated $\sD_X[\bs]$-module with grade number $j(N)=n+1$, and relative holonomic over $\bC[\bs]$. Then there exists an open affine subset $V=\Spec R$ in $\bC^r$ such that the intersection of   $V$ with each irreducible component of codimension one of $Z(B_N)$ is not empty, and
the  module $N\otimes_{\bC[\bs]}R$ is relative holonomic over $R$ and $(n+1)$-Cohen-Macaulay over $\sA_R$.
\end{prop}

Then one can apply the following Nakayama-type lemma, cf. \cite[Proposition 3.4.3]{BVWZ}:

\begin{prop}\label{propNK} 
Let $\Spec R$ be a nonempty open subset of $\bC^r$. Let $ N$ be an $\sA_R$-module that is relative holonomic over $R$ and $(n+l)$-Cohen-Macaulay over $\sA_R$ for some $0\le l\le r$. Then 
\[\bal\in Z(B_{ N})\; \text{\it  if and only if }  N\otimes_{R}\bC_\bal\not=0,\]
where $\bC_\bal$ is the residue field of the closed point $\bal\in \Spec R$.
\end{prop}

\subsection{} We let $F=(f_1,\ldots, f_r):X\ra\bC^r$ be a morphism. Let $\ba\in\bN^r$ be such that $\bf^\ba$ is not invertible. We consider now the left $\sD_X[\bs]$-module
$$M= \frac{\sD_X[\bs]\bf^{\bs}}{\sD_X[\bs]\bf^{\bs+\ba}}.$$
We will prove in Section \ref{secSl} the following:

\begin{theorem}\label{thmMais} $\;$
\begin{enumerate}[(i)]
\item The $\sD_X[\bs]$-module $M$ has grade number $j(M)=n+1$, and is relative holonomic over $\bC[\bs]$. 
\item Every irreducible component of $Z(B_{M})$ of codimension one is a hyperplane in $\bC^r$ of type $l_1s_1+\ldots +l_rs_r+b=0$ with $l_i\in \bQ_{\ge 0}$, $b\in \bQ_{>0}$, and for each such hyperplane there exists $i$ with $a_i\ne 0$ such that $l_i>0$.
\item Every irreducible component of $Z(B_{M})$ of codimension $>1$ can be translated by an element of $\bZ^r$ into a component of codimension one.
\end{enumerate}
\end{theorem}

For ${\sD_X[\bs]\bf^{\bs}}/{\sD_X[\bs]\bf^{\bs+\mathbf{1}}}$, Theorem \ref{thmMais} (i) and (iii)  are due to Maisonobe \cite[R\'esultat 2]{M}, and Part (ii) without the strict positivity of $l_i$ is due to Sabbah and Gyoza \cite{G}.

Granted this theorem, we can prove all the theorems from the introduction.

\subsection{Analytic case} All the above results hold in the local analytic case as well, by appropriately adapting the arguments, cf. \cite[3.6]{BVWZ}.

\subsection{Proof of Theorem \ref{thrmMoreA}.}\label{subsV} Let $M$ be as in Theorem \ref{thmMais}. Then
$$
B_M=B_{F,i},
$$
as in the introduction, since $M$ is a cyclic $\sD_X[\bs]$-module generated by the class of $\bf^{\bs}$. Hence the claim is equivalent to Theorem \ref{thmMais} (ii) and (iii). $\hfill\Box$

\subsection{Proof of Theorem \ref{thrmMoreC}  --  reduction.} It is enough to prove the claim for $B_{F,i}$ with $f_i$ not invertible. Indeed, this follows from (\ref{eqIII}) and (\ref{eqI}).

\subsection{Proof of Theorem \ref{thrmMoreC}  -- one inclusion.}\label{subsPT}

We prove first that $\cS_i(F)$ is a subset of $\Exp(Z(B_{F,i}))$. The proof is a slight generalization of the proof that $\cS(F)\subset \Exp(Z(B_F))$ from \cite{B-ls}. 

Take $\bla$ in $(\bC^*)^r$ but not in $\Exp(Z(B_{F,i}))$. Fix $\bal\in \Exp^{-1}(\bla)$. Then $\bal+\bk$ does not lie in $Z(B_{F,i})$ for any $\bk\in\bZ^r$. So, fixing  $\bk\in\bZ^r$, there exists $b(\bs)\in \bC[\bs]$ such that $b(\bal+\bk)\ne 0$ but
$$
b(\bal+\bk)\bf^{\bal+\bk} =P\cdot \bf^{\bal+\bk+\bbe_i}
$$
for some $P\in\sD_{X}$. Hence there is an equality $$\sD_X\cdot \bf^{\bal+\bk}= \sD_X\cdot\bf^{\bal+\bk+\bbe_i}$$ for all $\bk\in\bZ^r$, as $\sD_X$-submodules of $\sO_X[f^{-1}]\bf^\bal$ where $f=\prod_{j=1}^rf_r$. In particular, for any fixed $\bk\in\bZ^r$,
$$
\sD_X\cdot f_i^{-l}\bf^{\bal+\bk} =\sD_X\cdot f_i^{l}\bf^{\bal+\bk}
$$
for all integers $l\gg 0$. Thus for integers $k_j\gg 0$ for $j\ne i$ and $l\gg 0$, we have the equality
$$
\sD_X\cdot f_i^{\al_i-l} \prod_{j\ne i}f_j^{\al_j-k_j} = \sD_X\cdot f_i^{\al_i+l} \prod_{j\ne i}f_j^{\al_j-k_j}.
$$
Taking the analytic de Rham complex on both sides, we obtain using Proposition \ref{lemRH2} an isomorphism of perverse sheaves on $X$,
\begin{equation}\label{eqRj}
Rj_* L_{\bla^{-1}}[n] = h_{!*} Rg_* L_{\bla^{-1}}[n]
\end{equation}
where $g$ and $h$ are as in \ref{eqGH}, and $L_{\bla^{-1}}$ is the local system on $U$ defined as in \ref{subDTE}. Thus ${\bla^{-1}}$ is not in $\cS_i(F)$ by Proposition \ref{propNSe}. Then also $\bla$ is not in $\cS_i(F)$, by  Theorem \ref{thrmMoreB}. $\hfill\Box$

\subsection{Proof of Theorem \ref{thrmMoreC}  -- the other inclusion.}\label{subsPT2}  
Conversely, we show now that $\Exp(Z(B_{F,i}))$ is a subset of $\cS_i(F)$. 

By Theorem \ref{thrmMoreA}  it is enough to show that for a generic point $\bal$ of a hyperplane $L\cdot \bs+ b=0$ contained in  $Z(B_{F,i})$, the image $\Exp(\bal)$ is in $\cS_i(F)$.

Fix an integral $k>0$ divisible by $L\cdot \bbe_i$, which we know to be $>0$ by Theorem \ref{thrmMoreA},  and let $\bk=(k,\ldots,k)\in\bZ^r$. Define
$
l_0= (L\cdot \bk)/(L\cdot\bbe_i),
$
so that $l_0$ is a positive integer. Then
$$
L\cdot (\bs + \bk)+b = L\cdot(\bs + l_0\cdot \bbe_i) +b 
$$
and hence
$$
\bal-\bk \in (Z(B_{F,i}) -l_0\cdot \bbe_i).
$$
By \cite[Proposition 4.7]{B-ls}, for any $l\in\bZ_{>0}$ we have
$$
Z(B_{F}^{\; l\cdot \bbe_i}) = \bigcup_{l'=0}^{l-1} (Z(B_{F,i})-l'\cdot\bbe_i).
$$
Hence 
$$\bal-\bk\in Z(B_{F}^{\; l\cdot \bbe_i})
$$
for all $l> l_0$. Thus  
$
\bal-\bk$ is a generic point on a hyperplane in $Z(B_{F}^{\; l\cdot \bbe_i})
$
for all $l\gg k\gg \Vert \bal \Vert$ with $k$ divisible by $L\cdot\bbe_i$.

Consider the $\sD_X[\bs]$-module
$$
M=\frac{\sD_X[\bs]\bf^{\bs-\bk}}{\sD_X[\bs]\bf^{\bs-\bk+l\bbe_i}}.
$$ 
Since $M$ is a cyclic $\sD_X[\bs]$-module, $B_M$ is up to a shift $\bk$ in $\bs$ equal to $B_{F}^{\; l\cdot \bbe_i}$. Thus, $\bal$ is a generic point on a hyperplane in $Z(B_M)$.

Now we can  apply  Theorem \ref{thmMais} ${\color{hot} (i)}$ to $M$, since shifting the variables $\bs$ by $\bk$ is harmless. This gives that $M$ satisfies the conditions of Proposition \ref{propGeN}, and thus by Proposition \ref{propNK},  
$$
M\otimes_{\bC[\bs]}\bC_\bal\not\simeq 0,
$$
since $\bal$ is generic on a codimension-one irreducible component of $Z(B_M)$. Then, by Proposition \ref{propSEr} we have that $\Exp(\bal)$ is in $\cS_i(F)$.  $\hfill\Box$

\section{Refining Maisonobe's theorem}\label{secSl}

This section is devoted to the proof of 
Theorem \ref{thmMais} by extending the arguments from Maisonobe \cite[R\'esultat 2]{M}. We keep the same notation as in Theorem \ref{thmMais}. We also provide extensions from reduced characteristic varieties to characteristic cycles of some results of \cite{M} of independent interest.

\subsection{} 
Besides the relative filtration on $\sD_X[\bs]$, which we denote now by $F^{\textup{rel}}$, we will also use the filtration $F^{\sharp}$ which on $\sD_X$ agrees with the usual filtration by the order of the operators, and such that the order of all $s_i$ is 1. That is,
$$
F^{\textup{rel}}_p(\sD_X[\bs]) = (F_p\sD_X)[\bs]\quad\text{and}\quad F^\sharp_p (\sD_X[\bs])=\sum_{q+|\bi|=p}F_q\sD_X\cdot\bs^\bi.
$$
By 
$$\gr, \gr^{\textup{rel}}, \gr^\sharp,\quad  \Ch, \Ch^{\textup{rel}}, \Ch^\sharp,\quad \CC,  \CC^{\textup{rel}}, \CC^\sharp,$$
we will denote the operations of taking the associated graded objects, characteristic varieties, and characteristic cycles (that is, the sum of the irreducible components of the characteristic varieties together with their multiplicities) with respect to the filtrations 
$$F_\bullet \sD_X, F_\bullet ^{\textup{rel}}(\sD_X[\bs]), F_\bullet^{\sharp}(\sD_X[\bs]),$$ respectively. 

Note that for a finitely generated $\sD_X[\bs]$-module $N$, $\Ch^\sharp(N)$ is conic in $T^*X\times\bC^r$, that is, homogeneous along the fibers over the projection to $X$.

The following is straight-forward: 
\begin{lemma}\label{lem:pj}
Let $p_2:T^*X\times\bC^r\ra\bC^r$ be the second projection. For any finitely generated $\sD_X[\bs]$-module $N$, 
\[p_2(\Ch^\sharp(N))=\supp_{\bC[\bs]}(\gr^\sharp N).\]
\end{lemma}

\subsection{} We address a generalization of the specializations from \cite[\S 1.4-1.7]{Gil}. We consider a flat morphism $p\colon \mathscr X\ra S$ between two smooth irreducible varieties over $\bC$, and let $\mathbf{0}\in S$ be a point locally given by a regular sequence $\{s_1=\ldots=s_r=0\}$. Let $\sX_{\mathbf 0}=p^{-1}(\mathbf{0})$ and let $i\colon \sX_{\mathbf 0}\ra \sX$ be the closed embedding. We set $\sX^o=\sX\setminus\sX_\mathbf{0}$.

Let $\sN$ be a coherent $\sO_\sX$-module $\sN$. Define 
\[\lim_{\bs\to\mathbf{0}} \sN=[Li^*\sN ]\in K(\sX_\mathbf{0})\]
in the Grothendieck group of coherent $\sO_{\sX_{\mathbf 0}}$-modules. 




We let $\supp(\sN)$ denote the cycle given by the sum of  the irreducible components of the support of $\sN$ with multiplicities. We let $\supp_k(\sN)$ be the purely $k$-dimensional part.  One can extend this association to a surjective group homomorphism
$$
\supp: K_k(\sX)\ra Z_k(\sX)
$$
where $K_k(\sX)$ is the subgroup $K_k(\sX)$ of $K(\sX)$ generated by modules with support of dimension $\le k$, and $Z_k(\sX)$ is the group of pure $k$-dimensional cycles on $\sX$. 

For a reduced subvariety $Z$ of $\sX$ intersecting $\sX_\mathbf{0}$ properly, we define 
\be\label{eqLim}
\lim_{\bs\to\mathbf{0}}Z, 
\ee
its specialization at $\sX_{\mathbf 0}$, to be the algebraic cycle on $\sX_\mathbf{0}$ given by the scheme-theoretical intersection of $Z$ and $\sX_{\mathbf 0}$.  

Iteratively applying \cite[Proposition 1.5.3]{Gil} one obtains: 

\begin{lemma}\label{lem:rspecialization} Let $k\ge r$. Suppose that each irreducible component of $\supp(\sN)$ is purely of dimension $k$ and intersects $\sX_\mathbf{0}$ properly. Then, in $Z_{k-r}(\sX_0)$, 
\[\lim_{\bs\to\mathbf{0}}\supp(\sN)=\supp \lim_{\bs\to\mathbf{0}} \sN.\]
\end{lemma}

\subsection{} For a conic irreducible Lagrangian subvariety $\Lambda\subset T^*U$,  define an $(n+r)$-dimensional subvariety of $T^*U\times \bC^r$ by
\[\Lambda^\sharp\coloneqq \{(x,\xi+\sum_{i=1}^r \al_id\log f_i (x),\bal)| (x,\xi)\in \Lambda, \bal=(\al_1,\ldots,\al_r)\in\bC^r\}.\]
This operation can be naturally generalized to conic Lagrangian cycles in $T^*U$.

\begin{theorem}\label{thm:char} For any $\ba\in \bZ^r$, there is an equality of cycles on $T^*X\times\bC^r$,
$$\CC^\sharp(\sD_X[\bs]\bf^{\bs+\ba})=\overline{(T^*_UU)^\sharp},$$
where $T^*_UU$ is the zero section of $T^*U$ and the closure is taken inside $T^*X\times \bC^r$.
\end{theorem}
\begin{proof} The equality for the underlying sets is proved in \cite{KK}, \cite{BMM3}. When $r=1$, the equality for cycles is \cite[Proposition 2.5]{Gil}. The arguments in \cite{BMM3} indeed prove the equality for cycles in general. 
\end{proof} 

We write $\overline{(T^*_UU)^\sharp}|_{\bs=\mathbf{0}}$ for the algebraic cycle of the scheme theoretical intersection of $\overline{(T^*_UU)^\sharp}$ and $p_2^{-1}(\mathbf{0})$. The following generalizes \cite[Theorem 3.2]{Gil}; the statement for characteristic varieties is due to \cite{BMM-cons}.



\begin{coro}\label{thm:char1} There is an equality of cycles on $T^*X$,
$$\CC(\sO_X(*D))=\lim_{\bs\to \mathbf{0}} \overline{(T^*_UU)^\sharp}=\overline{(T^*_UU)^\sharp}|_{\bs=\mathbf{0}}.$$
\end{coro}
\begin{proof} The second equality is just the statement that every irreducible component of $\overline{(T^*_UU)^\sharp}$ intersects $p_2^{-1}(\mathbf{0})$ properly. By definition, the support of $\overline{(T^*_UU)^\sharp}$ is irreducible, being the closure of the graph of a function defined on $U\times\bC^r$. Hence it is of dimension $n+r$. We will simply refer to \cite{BMM-cons} for the fact that it intersects the special fiber properly.

Now, for some fixed $k>0$, we let $\bk=(k,\ldots, k)$ and we fix a filtered free resolution 
\[\cP^\bullet \ra \sD_X[\bs]\bf^{\bs-\bk}\]
with respect to the $\sharp$-filtration. By definition, $\cP^k$ are finite direct sums of free $\sD_X[\bs]$-modules of rank one equipped with the filtration $F^\sharp$ possibly up to a shift, and the differentials are strict with respect to the filtration; see \cite[A.IV, Proposition 4.1]{Bj} where is it also shown that filtered free resolutions exists
 at least locally on $X$. 
 
By strictness, we obtain a free resolution of the associated graded module
\[\gr^\sharp \cP^\bullet\ra \gr^\sharp (\sD_X[\bs]\bf^{\bs-\bk}).\]
Let $$i\colon T^*X\simeq T^*X\times \{\mathbf 0\}\to T^*X\times \bC^r$$ be the closed embedding. Then
the complex $Li^*\gr^\sharp (\sD_X[\bs]\bf^{\bs-\bk})$ is quasi-isomorphic to $i^*\gr^\sharp \cP^\bullet$. Thus
$$
\lim_{\bs\ra\mathbf{0}} \gr^\sharp (\sD_X[\bs]\bf^{\bs-\bk}) =[i^*\gr^\sharp \cP^\bullet]
$$
in $K_n(T^*X)$.  By Theorem \ref{thm:char} and the above discussion, the support of $\gr^\sharp (\sD_X[\bs]\bf^{\bs-\bk})$ intersects $p_2^{-1}(0)$ properly. Then we can apply Lemma \ref{lem:rspecialization} to obtain that
$$
\lim_{\bs\ra\mathbf{0}}  \overline{(T^*_UU)^\sharp} = \supp\; [i^*\gr^\sharp \cP^\bullet]
$$
as $n$-cycles in $T^*X$.

On the other hand, while $i^*$ and $\gr^\sharp$ do not necessarily commute, they do so in $K(T^*X)$, that is,
$$
[i^*\gr^\sharp \cP^\bullet]=[\gr (i^*\cP^\bullet)].
$$
To see this, note that the derived pullback of the filtered complex $(\cP^\bullet, F^\sharp_\bullet \cP^\bullet)$ gives us a filtered complex of $\sD_X$-modules and hence a convergent spectral sequence 
\[ H^\bullet(i^*\gr^\sharp \cP^\bullet)\Rightarrow \gr ( H^\bullet(i^*\cP^\bullet)),\]
see for example \cite[\S 3]{Lau}.

Now, by \cite[Corollary 5.4, Theorem 1.3]{WZ}, we know that for $k\gg0$
\[Li^*(\sD_X[\bs]\bf^{\bs-\bk})\simeq \sO_X(*D).\]
Therefore, we have 
\[[\gr (i^*\cP^\bullet)]=[\gr(\sO_X(*D))],\]
the support of which as a cycle is $\CC(\sO_X(*D)$.
\end{proof}

\begin{rmk}
Suppose $N$ is a regular holonomic $\sD_X$-module and consider its localization $N(*D)$ along $D$, which is also regular holonomic. We can assume that $N(*D)$ is generated by a coherent $\sO_X$-submodule $N_0$,  that is, $N(*D)=\sD_X\cdot N_0$. Then we have the coherent $\sD_X[\bs]$-module $\sD_X[\bs]\bf^\bs\cdot N_0$. By \cite{BMM-cons} and \cite[Corollary 5.4]{WZ}, Theorem \ref{thm:char} and Corollary \ref{thm:char1} hold in this more generalized setting, proven for characteristic varieties in \cite{BMM-cons}, that is, 
\[\CC^\sharp(\sD_X[\bs]\bf^\bs\cdot N_0)=\overline{\Lambda_U^\sharp} \textup{ and } \CC(N(*D))=\lim_{\bs\to \mathbf{0}}\overline{\Lambda_U^\sharp}=(\overline{\Lambda_U^\sharp})|_{\bs=\mathbf{0}},\]
where $\Lambda_U=\CC(N|_U)$. 
\end{rmk}

The next theorem slightly generalizes the second statement of \cite[R\'esultat 6]{M} to characteristic cycles:
\begin{theorem}\label{thrmGenM6}
For every $\ba\in \bZ^r$,
$$\CCr(\sD_X[\bs]\bf^{\bs+\ba})=\CC(\sO_X(*D))\times \bC^r.$$
\end{theorem} 
\begin{proof}
We prove the case $\ba=\mathbf{0}$, the general case is similar. By \cite{BMM-cons} and \cite{M},  the required equality is true for the underlying sets.  We now prove that they are the same as cycles. 

We first observe that 
\begin{equation}\label{eq:limchrel}
\CC(\sD_X[\bs]\bf^\bs\otimes^L_{\bC[\bs]} \bC_{\bal})=\lim_{\bs\to \bal}\CCr(\sD_X[\bs]\bf^{\bs})
\end{equation}
for every $\bal\in \bC^r$, where $\bC_{\bal}=\bC[\bs]/(\bs-\bal)$ is the residue field at $\bal$. This follows by the same argument as in the proof of Corollary \ref{thm:char1}, replacing $\sharp$-filtration with the relative filtration and specializing to $p_2^{-1}(\bal)$.

By \cite[Corollary 5.4, Theorem 1.3]{WZ}, we have for integral $k\gg 0$ and $\bk=(k,\ldots,k)$,
\[\sD_X[\bs]\bf^\bs\otimes^L_{\bC[\bs]}\bC_{-\bk}\simeq \sO_X(*D)\]
and hence 
\[\CC(\sO_X(*D))=\lim_{\bs\to -\bk}\CCr(\sD_X[\bs]\bf^{\bs}).\]
Since we know the required equality is true for the underlying sets, the above limit as $\bs\to -\bk$ is just taking the fiber at $\bs=-\bk$. Then the required equation for cycles follows.
\end{proof}

We will use the following due to Maisonobe \cite[Proposition 14]{M}:
\begin{prop}\label{prop:npure} Let $\ba\in \bZ^r$. Every non-zero { $\sD_X[\bs]$-submodule} of $\sD_X[\bs]\bf^{\bs+\ba}$ has grade number equal to $n$. Equivalently,  $\sD_X[\bs]\bf^{\bs+\ba}$ is $n$-pure over $\sD_X[\bs]$.
\end{prop}

Next statement for characteristic varieties for the case $\ba=\mathbf{1}$ is due to \cite{BMM3}: 

\begin{coro}\label{cor:char} Let $\ba\in\bN^r$ such that $\bf^\ba$ is not invertible. Then
$$\CC^\sharp\left(\dfrac{\sD_X[\bs]\bf^\bs}{\sD_X[\bs]\bf^{\bs+\mathbf{a}}}\right)=\overline{(T^*_UU)^\sharp}|_{\{\bf^\ba=0\}}.$$
\end{coro}
\begin{proof} Let $b=\bf^\ba$. By $\{b=0\}$ in the statement we mean the effective divisor on $T^*X\times\bC^r$, with possibly non-trivial multiplicities, defined by $b$.  Note that $\overline{(T^*_UU)^\sharp}|_{\{b=0\}}$ is a well-defined cycle, since  $b$ restricts to a non-trivial Cartier divisor on $\overline{(T^*_UU)^\sharp}$, and it equals the limit construction from (\ref{eqLim}) applied to the flat map $b:T^*X\times\bC^r\ra \bC$.

Write $N=\sD_X[\bs]\bf^\bs$ and $N_1=\sD_X[\bs]\bf^{\bs+\ba}$  for this proof. By Proposition \ref{prop:npure}, $N$ is $n$-pure over $\sD_X[\bs]$. We apply the argument similar to the proof of \cite[Lemma 3.4.2]{BVWZ}. The outcome is that the $\gr^\sharp$ of $N_1\to N$ is given by 
$$ 
\gr^\sharp N \xrightarrow{b}  \gr^\sharp N,
$$
which is injective (since $\gr^\sharp N$ is pure by  \cite[Proposition 4.4.1]{BVWZ}). Considering the induced filtration on $N/N_1$, the injectivity of $\gr^\sharp N \xrightarrow{b}  \gr^\sharp N$ gives an isomorphism
\be\label{eqNbN}
\gr^\sharp(N/N_1)\simeq (\gr^\sharp N)/b(\gr^\sharp N),
\ee
see also \cite[(3.3)]{BVWZ}.

Then we obtain equalities of cycles
$$
\overline{(T^*_UU)^\sharp}|_{\{b=0\}} = \CC^\sharp(N)|_{\{b=0\}} = \supp ((\gr^\sharp N)|_{\{b=0\}}) = \supp (\gr^\sharp(N/N_1)),
$$
where the first equality is by Theorem \ref{thm:char},  the second equality is by Lemma \ref{lem:rspecialization} which in this case is the same as \cite[Proposition 1.5.3]{Gil}, and the third equality is by (\ref{eqNbN}).
\end{proof}

\begin{prop}\label{propDMA1}
Let $\ba\in\bN^r$ such that $\bf^\ba$ is not invertible. Let 
$$
M= \frac{\sD_X[\bs]\bf^{\bs}}{\sD_X[\bs]\bf^{\bs+\ba}}.
$$
Then $j(M)=n+1$ and $M$ is relative holonomic over $\bC[\bs]$.
\end{prop}
\begin{proof}
By Theorem \ref{thrmGenM6}, $\sD_X[\bs]\bf^\bs$ is relative holonomic over $\bC[\bs]$. Since every nonzero subquotient  is also relative holonomic by \cite[Lemma 3.2.4]{BVWZ}, it follows that $M$ is also regular holonomic.

By \cite[A.IV 4.15]{Bj}, the grade number  of $M$ over $\sD_X[\bs]$ is the grade number of $\gr^\sharp M$ over $\gr^\sharp(\sD_X[\bs])$. Since the latter is the codimension of $\Ch^\sharp(M)$ in $T^*X\times\bC^r$, we have $j(M)=n+1$ by Corollary \ref{cor:char}. 
\end{proof}

\subsection{Maximal pure tame extension.}\label{subMPTE}  For every finitely generated $\sD_X[\bs]$-submodule $N$ of $\sD_X[\bs,1/f]\bf^\bs$ there exists some $\bk\in \bZ_{\ge0}^r$ such that \[N\subseteq \sD_X[\bs]\bf^{\bs-\bk}.\] Thus by Proposition \ref{prop:npure}, all $N$ are $n$-pure. Consider the family of all finitely generated $\sD_X[\bs]$-submodules $N$ of $\sD_X[\bs,1/f]\bf^\bs$ such that $\sD_X[\bs]\bf^\bs\subset N$ and the grade number $$j(N/\sD_X[\bs]\bf^\bs)\ge n+2.$$ This family always has a unique maximal member $\cL$, called the {\it maximal tame pure extension} of $\sD_X[\bs]\bf^\bs$, by \cite[Proposition A.IV 2.10]{Bjo}.  

We extend slightly the properties of $\cL$ proven in \cite{M}.

\begin{lemma}\label{lemMPT0} Let $\ba\in\bN^r$ such that $\bf^\ba$ is not invertible.
Let
\[t\colon \sD_X[\bs,1/f]\bf^\bs\ra \sD_X[\bs,1/f]\bf^\bs.\]
be the action defined by substituting $\bs$ by $\bs+\ba$. Then:
\begin{enumerate}[(i)]
\item $t \cL\subset \cL$,
\item\label{Ai} ${\cL}/{t \cL}$ is $(n+1)$-pure over $\sD_X[\bs]$ and relative holonomic over $\bC[\bs]$.
\end{enumerate}
\end{lemma}
\begin{proof}
(i) We have that $t \cL$ is the maximal tame pure extension of $\sD_X[\bs]\bf^{\bs+\ba}$ and hence {over $\sD_X[\bs]$} the grade number $j(t \cL/\sD_X[\bs]\bf^{\bs+\ba})\ge n+2$. Since 
\[\frac{t \cL}{\sD_X[\bs]\bf^{\bs+\ba}}\twoheadrightarrow \dfrac{t \cL}{t \cL\bigcap\sD_X[\bs]\bf^{\bs}}\simeq  \dfrac{t \cL+\sD_X[\bs]\bf^{\bs}}{\sD_X[\bs]\bf^{\bs}} ,\]
and the grade number of a quotient can only increase (cf. \cite[Theorem 3.2.2]{BVWZ}), we have $$j\left(\dfrac{t \cL+\sD_X[\bs]\bf^{\bs}}{\sD_X[\bs]\bf^{\bs}}\right) \ge n+2.$$  By maximality, $t \cL+\sD_X[\bs]\bf^{\bs}\subset \cL$ and hence $t \cL\subset\cL$.

(ii) Since $\cL/t\cL$ is a subquotient of $\sD_X[\bs]\bf^{\bs-\bk}$ for some $\bk\in\bZ^r$, it is relative holonomic over $\bC[\bs]$ by Theorem \ref{thrmGenM6} and by \cite[Lemma 3.2.4]{BVWZ}. 

Next we prove that every nonzero $\sD_X[\bs]$-submodule of $\cL/t \cL$ has grade number $\le n+1$. Assume on the contrary there exists $\cL'$ such that
\[t \cL \subset \cL'\subset \cL\]
and $j(\cL'/t\cL)\ge n+2$. Since $\cL\subset t^{-1}\cL'$, we have $j(t^{-1}\cL'/\sD_X[\bs]\bf^\bs)\le n+1$ by maximality. Considering the short exact sequence
\[0\ra \dfrac{\cL}{\sD_X[\bs]\bf^\bs}\ra\dfrac{t^{-1}\cL'}{\sD_X[\bs]\bf^\bs}\ra\dfrac{t^{-1}\cL'}{\cL}\ra0\]
we hence have $j(t^{-1}\cL'/\cL)\le n+1$, applying \cite[Theorem 3.2.2]{BVWZ} again. However, since $t^{-1}\cL'/\cL\simeq \cL'/t \cL$, we get a contradiction. 

By \cite[Proposition A. IV. 2.11]{Bjo}, not all quotients $t^l\cL/t^{l+1} \cL$ with $l\ge 0$ have the same grade number as $\cL$. Since $t^l\cL/t^{l+1} \cL\simeq \cL/t\cL$ for all $l\ge 0$ as $\sD_X[\bs]$-modules via a translation in $\bs$, we then must have $j(\cL/t\cL)\ge n+1$. 
\end{proof}

\begin{lemma}\label{lemMPT}
With the same notation as in the previous lemma, we have:
\begin{enumerate}[(i)]
\item  $\CC^\sharp(\cL)=\CC^\sharp(\sD_X[\bs]\bf^{\bs}).$
\item \[\CC^\sharp(\cL/t\cL)=\overline{(T^*_UU)^\sharp}|_{\{\bf^\ba=0\}}=\CC^\sharp\left(\dfrac{\sD_X[\bs]\bf^\bs}{\sD_X[\bs]\bf^{\bs+\ba}}\right).\]
\item Every irreducible component of the zero loci of $$\Ann_{\bC[\bs]}(\cL/t\cL)\quad\text{and}\quad\Ann_{\bC[\bs]}(\gr^\sharp(\cL/t\cL))$$ is a hyperplane the type $l_1s_1+\ldots l_rs_r+b=0$ with $l_i\in\bN$.
\end{enumerate}
\end{lemma}
\begin{proof} (i) Let $\bk\in \bZ^r$ be such that $\sD_X[\bs]\bf^{\bs}\subset \cL\subset\sD_X[\bs]\bf^{\bs-\bk}$. 
By Theorem \ref{thm:char}, we know $\CC^\sharp(\sD_X[\bs]\bf^{\bs-\bk})=\CC^\sharp(\sD_X[\bs]\bf^\bs).$
This implies the claim.

(ii)  Since $t\cL\subset \cL$, it follows that $t\cL=\bf^\ba\cL$. Since $\CC^\sharp(\cL)=\CC^\sharp(\sD_X[\bs]\bf^\bs)$,  applying the same proof as for Corollary \ref{cor:char} for the multiplication map $\bf^\ba:\cL\ra\cL$ gives the claim.

(iii) 
By \cite[Theorem A:IV 4.15]{Bjo}, $\gr^{\textup{rel}}(\cL/t\cL)$ is also $(n+1)$-pure after choosing a suitable good filtration. Then purity in the commutative case gives that $\Chr(\cL/t\cL)$ is pure of dimension $n+r-1$, cf. \cite[Theorem A:IV 3.7]{Bjo}. Since $\cL/t\cL$ is a subquotient of $\sD_X[\bs]\bf^{\bs-l\bbe_i}$, it is relative holonomic over $\bC[\bs]$ by Theorem \ref{thrmGenM6} and \cite[Lemma 3.2.4]{BVWZ}. Thus we can apply \cite[Lemma 3.4.1]{BVWZ} to obtain that $$Z(B_{\cL/t\cL})=p_2(\Chr(\cL/t\cL))$$ is pure of codimension 1. Now we know by \cite{Sab, G}, with $\cL$ instead of $\sD_X[\bs]\bf^\bs$, that each hyperplane in $Z(B_{\cL/t\cL})$ has slopes in $\bN^r$ as required. Since the zero locus of $\Ann_{\bC[\bs]}(\gr^\sharp(\cL/t\cL))$ is contained in the zero locus of the initial term of the product of polynomials of degree 1 defining $Z(B_{\cL/t\cL})$, the last claim follows as well.
\end{proof}

\begin{definition} We denote by $$\cH_\ba, \cH_\ba^\sharp, \cH_\ba(F), \cH_\ba^\sharp(F),$$
the sets of primitive slopes $L=(l_1,\ldots,l_r)\in\bN^r$ of hyperplanes of type $L\cdot \bs+b=0$ appearing as irreducible components of the support over $\bC[\bs]$ of
$$
\cL/t\cL, \gr^\sharp(\cL/t\cL), \dfrac{\sD_X[\bs]\bf^\bs}{\sD_X[\bs]\bf^{\bs+\ba}}, \gr^\sharp\left(\dfrac{\sD_X[\bs]\bf^\bs}{\sD_X[\bs]\bf^{\bs+\ba}}\right),
$$ 
respectively. Note that all four slope sets do not change if $\ba$ is replaced by a multiple $l\ba$ with $l\in\bZ_{>0}$.
\end{definition}

The following extends slightly \cite[R\'esultat 6]{M}:

\begin{prop}\label{prop:islope} With the same notation as in Lemma \ref{lemMPT0},
$$
\cH_\ba^\sharp(F) = \cH_\ba^\sharp=\cH_\ba =  \cH_\ba(F).
$$
\end{prop}
\begin{proof} The first equality is by Lemma \ref{lemMPT} (ii).

Since $t^{l}\cL\subset\sD_X[\bs]\bf^{\bs}\subset\cL\subset \sD_X[\bs]\bf^{\bs-l\ba}$ for some $l\in \bZ_{\ge 0}$, we have 
\be\label{eqSupp1}
\supp_{\bC[\bs]}\left(\dfrac{\cL}{t^l \cL}\right)\subset \supp_{\bC[\bs]}\left(\dfrac{\sD_X[\bs]\bf^{\bs-l\ba}}{\sD_X[\bs]\bf^{\bs}}\right)\cup \supp_{\bC[\bs]}\left(\dfrac{\sD_X[\bs]\bf^{\bs}}{\sD_X[\bs]\bf^{\bs+l\ba}}\right)
\ee
and 
\be\label{eqSupp2}\supp_{\bC[\bs]}\left(\dfrac{\sD_X[\bs]\bf^{\bs-l\ba}}{\sD_X[\bs]\bf^{\bs}}\right)\subset \supp_{\bC[\bs]}\left(\dfrac{t^{-l}\cL}{\cL}\right)\cup \supp_{\bC[\bs]}\left(\dfrac{\cL}{t^{l}\cL}\right).\ee
Since $\bs\cdot t^l=t^l\cdot \bs-l\ba$ for every $l\in\bZ$, both supports on the right-hand side of (\ref{eqSupp1}) have the same slope set $\cH_{l\ba}(F)=\cH_\ba(F)$. Hence $\cH_\ba$, the slope set of $\cL/t^l\cL$, is included in $\cH_\ba(F)$. From (\ref{eqSupp1}), we get now the third claimed equality $\cH_\ba=\cH_\ba(F)$.

It remains to prove the second equality.

If $\mathfrak{p}$ is an associated prime of $\gr^{\textup{rel}}({\cL}/{t\cL})$, then 
\[\mathfrak{p}=\mathfrak{p}_X\otimes \mathfrak{p}_{\bC^r}\]
where $\mathfrak{p}_{\bC^r}$ is an ideal generated by some hyperplane in $\bC^r$ with   slope $L\in \cH_\ba$, and $\mathfrak{p}_X$ is the prime ideal in $T^*X$ of an irreducible Lagrangian subvariety, by relative holonomicity of $\cL/t\cL$ from Lemma \ref{lemMPT0}. Then we define 
\[in(\mathfrak{p})\coloneqq \mathfrak{p}_X\otimes in(\mathfrak{p}_{\bC^r})\]
where $in(\mathfrak{p}_{\bC^r})$ is the initial homogenous ideal generated by $L\cdot\bs$.  We claim that the set of associated primes of $\gr^\sharp({\cL}/{t_i\cL})$ is the union of all $in(\mathfrak{p})$.
Together with Lemma \ref{lem:pj}, the claim implies $\cH_\ba=\cH^\sharp_\ba$.

Now we prove the last claim. We assume $b_\cL$ is the generator of the radical ideal of $B_{\cL/t\cL}$. Since the support of $Z(B_{\cL/t\cL})$ is pure of codimension 1, we know that 
\[b_\cL=\prod_{L\in \cH_\ba} (L\cdot \bs+r_L)\]
for some $r_L$. If $b_\cL$ is of degree 1, then the claim follows since $\Ch^\sharp(\cL/t_i\cL)$ is pure of dimension $n+r-1$. In general, we pick the smallest $k$ so that $b^k_\cL\in B_{\cL/t\cL}$. Picking a slope $L\in \cH_\ba$, we consider the short exact sequence 
\[0\ra \frac{b^k_{\cL}}{L\cdot \bs+r_L}\cdot \cL/t\cL\ra \cL/t\cL\ra \cQ\to 0,\]
where $\cQ$ is the quotient module. Doing induction on the degree of $b^k_\cL$, the proof of the claim is done. 
\end{proof}

\begin{prop}\label{propDMA3}
Let $\ba\in\bN^r$ such that $\bf^\ba$ is not invertible. Let 
$$
M= \frac{\sD_X[\bs]\bf^{\bs}}{\sD_X[\bs]\bf^{\bs+\ba}}.
$$
Then every irreducible component of $Z(B_M)$ of codimension $>1$ can be translated by an element of $\bZ^r$ into a component of codimension one. 
\end{prop}
\begin{proof}
First note that $\bs\cdot t=t\cdot \bs-\ba$ implies that
$$\supp_{\bC[\bs]}\left(\frac{t\cL}{t^2\cL}\right)=T\left(\supp_{\bC[\bs]}\left(\frac{\cL}{t\cL}\right)\right),$$
where $T:\mathbb{C}^r\to \mathbb{C}^r$ is translation by $\ba$. Then the exact sequence
$$0\to \frac{t\cL}{t^2\cL}\to \frac{\cL}{t^2\cL}\to \frac{\cL}{t\cL}\to 0$$
implies that 
$$\supp_{\bC[\bs]}\left(\frac{\cL}{t^2\cL}\right)= \supp_{\bC[\bs]}\left(\frac{\cL}{t\cL}\right)\cup T\left(\supp_{\bC[\bs]}\left(\frac{\cL}{t\cL}\right)\right).$$
Iterating this argument we obtain that
$$\supp_{\bC[\bs]}\left(\frac{\cL}{t^l\cL}\right)=\bigcup_{j=0}^{l-1}T^j\left((\supp_{\bC[\bs]}\left(\frac{\cL}{t\cL}\right)\right)$$ and $$\supp_{\bC[\bs]}\left(\frac{t^{-l}\cL}{\cL}\right)= \bigcup_{j=-l+1}^{0}T^j\left(\supp_{\bC[\bs]}\left(\frac{\cL}{t\cL}\right)\right).$$
With this description of the supports in mind, (\ref{eqSupp2}) implies that all components of $Z(B_M)$ are contained in translates of the some of the components of $Z(B_{\cL/t\cL})$. On the other hand, (\ref{eqSupp1}) implies that all of the components of $Z(B_{\cL/t\cL})$ are contained in translates of the codimension $1$ components of $Z(B_M)$. These two statements combined imply the claim.
\end{proof}

\begin{lemma}\label{logres}
Let $\mu:Y\to X$ be a log-resolution of the pair (X,D) that is an isomorphism above $U=X\setminus D$. Let $g_i=f_i\circ \mu$ and $G=(g_1,g_2,\dots, g_r)$. Let $\ba\in\bN^r$ such that $\bf^\ba$ is not invertible. Then 
$$\cH_\ba(F)\subset\cH_\ba(G).$$
\end{lemma}
\begin{proof}  We write $H_\ba(F)$ for the union of all the hyperplanes defined by the slopes in $\cH_\ba(F)$. Clearly it is enough to show that $H_\ba(F)\subset H_\ba(G)$.

We have shown above that $H_\ba(F)$ equals the reduced support over $\bC[\bs]$ of
$$\gr^\sharp\left(\dfrac{\cL}{t\cL}\right)\quad\text{and}\quad\gr^\sharp\left(\dfrac{\sD_X[\bs]\bf^\bs}{\sD_X[\bs]\bf^{\bs+\ba}}\right).$$

Let $(\mathbf{x}_j,{\bal}_j)=(x_{1,j},\dots,x_{n,j},\al_{1,j},\dots,\al_{r,j})$ with $j\ge 1$ be a sequence of points in $U\times \bC^r$ such that 
\begin{enumerate}
\item $\lim \mathbf{x}_j$ exists and it is a point of $\{\bf^\ba=0\}$, and
\item $\lim \left(\mathbf{x}_j, (d\log \bf^{\bal_j})(\mathbf{x}_j),\mathbf{\bal}_j \right) \text{ converges in }T^*X\times\bC^r.$
\end{enumerate}
By Lemma \ref{lem:pj} and Corollary \ref{cor:char}, we have for any such sequence that $\lim \bal_j\in H_i(F)$, and conversely, for any $\bal\in H_i(F)$ there exists such a sequence with $\lim \bal_j=\bal$. 

Write $$\omega_j=\sum_{l=1}^r \al_{l,j}\cdot d\log f_l=d\log \bf^{\bal_j}.$$
 
 Since $\mu$ is an isomorphism over $U$ and a proper map, there is a subsequence $\mathbf{x}_{j_k}$ of $\mathbf{x}_j$, such that $\mu^{-1}(\mathbf{x}_{j_k})$ converges on $Y$. Replace $\mathbf{x}_j$  by this subsequence, and $\bal_j$ by the corresponding subsequence consisting of the $\bal_{j_k}$. Clearly $(\mathbf{x}_j,\bal_j)$ still satisfies conditions $(1)$ and $(2)$ above. Denote $y=\lim \mu^{-1}(\mathbf{x}_j)$. Clearly $\mathbf{g}^\ba(y)=0$. 
 
Let $D'=\mu^{-1}(D)$. Choose coordinates on a small open $V$ around $y$ and trivializations of the cotangent bundles fitting in a commutative diagram of isomorphisms:
\[
\begin{tikzcd}
T^*(\mu(V\setminus D'))\arrow[r,"\mu^*"]\arrow[d] & T^*(V\setminus D')\arrow[d]\\
\mu(V\setminus D')\times \mathbb{C}^n\arrow{r} & (V\setminus D')\times \mathbb{C}^n.
\end{tikzcd}
\] 
For $\omega_j(\mathbf{x}_j)\in T^*(\mu(V\setminus D'))$, denote by $(\mathbf{x}_j,p_j)$ the corresponding point in $\mu(V\setminus D')\times\mathbb{C}^n$. 
Then under this diagram,
\[
\xymatrix{
\omega_j(\mathbf{x}_j) \ar@{|->}[r] \ar@{|->}[d]& \mu^*_{\mu^{-1}(\mathbf{x}_j)}\omega_j(\mathbf{x}_j)\ar@{|->}[d]\\
(\mathbf{x}_j,p_j)\ar@{|->}[r] & (\mu^{-1}(\mathbf{x}_j),M(\mathbf{x}_j)p_j)
}
\]
where
$$M(\mathbf{x})=\begin{pmatrix}
\frac{\partial \mu_1}{\partial y_1}|_{\mu^{-1}(\mathbf{x})} & \cdots & \frac{\partial \mu_n}{\partial y_1}|_{\mu^{-1}(\mathbf{x})}\\
\vdots & \ddots & \vdots\\
\frac{\partial \mu_1}{\partial y_n}|_{\mu^{-1}(\mathbf{x})} & \cdots & \frac{\partial \mu_n}{\partial y_n}|_{\mu^{-1}(\mathbf{x})}\\
\end{pmatrix}=Jac_\mu(\mu^{-1}(\mathbf{x})).$$
Then
$$\lim_{j\to \infty}(\mu^{-1}(\mathbf{x}_j),M(\mathbf{x}_j)p_j)=(y,Jac_{\mu}(y)\lim_{j\to \infty}p_j),$$
which shows that this limit exists, since $\lim_{j\to \infty}p_j$ exists by condition (2) above. Combining this with the fact that
$$\mu^*_{\mu^{-1}(\mathbf{x}_j)}\omega_j(\mathbf{x}_j)=\sum_{l=1}^r \al_{l,j}\cdot (d\log g_l)(\mu^{-1}(\mathbf{x_j})),$$
it follows that the sequence $(\mu^{-1}(\mathbf{x}_j),\bal_j)$ satisfies:
 $$\lim_{j\to \infty}\left(\mu^{-1}(\mathbf{x}_j), \sum_{l=1}^r s_{l,j}d\log(g_l)(\mu^{-1}(\mathbf{x_j})),\bal_j \right) \text{ converges in }T^*Y\times\bC^r.$$
which concludes the proof.
\end{proof}

\begin{prop}\label{propDMA}
Let $\ba\in\bN^r$ such that $\bf^\ba$ is not invertible. Let 
$$
M= \frac{\sD_X[\bs]\bf^{\bs}}{\sD_X[\bs]\bf^{\bs+\ba}}.
$$
Then: 
\begin{enumerate}[(i)]
\item Every irreducible component of $Z(B_M)$ of codimension one is a hyperplane in $\bC^r$ of type $l_1s_1+\ldots+l_rs_r+b=0$ with $l_j\in\bQ_{\ge 0}$ for all $1\le j\le r$, and $b\in\bQ_{>0}$. 
\item Moreover, for each such component there exists $j$ with $a_j\ne 0$ such that $l_j>0$.
\end{enumerate}
\end{prop}
\begin{proof}
Part (i) is due to \cite{Sab, G}. The strict positivity in part (ii) is new. Let $\mu:Y\to X$ be a log-resolution of the pair $(X,D)$,  let $g_j=\mu^*f_j$, and let $G=(g_1,\dots,g_r)$. By Lemma \ref{logres}, $\cH_\ba(F)\subset \cH_\ba(G)$. Hence it suffices to prove the statement for the tuple $G$ and locally at a point $y\in Y$ above $x$, since the global Bernstein-Sato ideal is the intersection of the local ones.

Chose coordinates on small open ball $V$ centered at $y$ and write $$g_j=u_jy_1^{l_{j,1}}\ldots y_n^{l_{j,n}}$$ where $u_j$ is a unit on $V$ and  $l_{j,k}\in \mathbb{Z}_{\geq 0}$.  Write $$L_k=(l_{1,k},\ldots, l_{r,k}).$$ Thus $l_{j,k}>0$ if and only if the divisor $\{y_k=0\}$ is  a component of $\{g_j=0\}$ in $V$. Set 
$$K= \bigcup_{j \text{ with } a_j\ne 0}  \{k \mid  l_{j,k}>0 \}.$$ 
By assumption, $K$ is non-empty.  Note that $L_k\cdot \ba>0$ for every $k\in K$ and that locally at $y$
\begin{align*}
\mathbf{u}^\bs\partial_{y_n}^{L_{n}\cdot\ba}\cdots \partial_{y_1}^{L_{1}\cdot\ba}\mathbf{u}^{-\bs}\cdot \mathbf{g}^{\bs+\ba}=\prod_{k\in K}\prod_{j=1}^{L_k\cdot \ba} (L_k\cdot \mathbf{s}+j) \mathbf{g}^\bs.
\end{align*}
This shows that 
$$\prod_{k\in K}\prod_{j=1}^{L_{k}\cdot\ba} (L_k\cdot \mathbf{s}+j)\in B_{N,y}$$
where 
$$N=\frac{\sD_Y[\bs]\mathbf{g}^{\bs}}{\sD_Y[\bs]\mathbf{g}^{\bs+\ba}}.$$
It follows that locally at $y$,
$$\mathcal{H}_\ba(G)\subset \{L_k\mid k\in K\}$$
as claimed. (One can further show that the last inclusion is an equality of sets.)
\end{proof}


\subsection{Proof of Theorem \ref{thmMais}.} This is now covered by Propositions   \ref{propDMA1}, \ref{propDMA3}, \ref{propDMA}.  $\hfill\Box$

\section{Proof of Proposition \ref{lemRH2}}\label{secRH2}

The proof is an adjustment of that of \cite[Theorem 5.4]{WZ}. 

\subsection{} We keep the notation as in Proposition \ref{lemRH2}. We write 
$$\cN_{k,l,\bal}\coloneqq\sD_X[\bs]\bf^{\bs+\bal-\bk+l\bbe_i}\subseteq j_*(\cM_\bla[\bs]\bf^{\bs})$$
and set $$\cN\coloneqq\cN_{0,0,\bal}$$
where 
\[\cM_\bla[\bs]\bf^\bs= \cM_\bla\otimes_{\sO_U}\sO_U[\bs]\bf^\bs\]
with the natural left $\sD_X[\bs]$-module structure.
Then it is enough to prove that 
 $$
\cN_{k,l,\bal}\otimes^L_{\bC[\bs]}{\bC_{\mathbf{0}}}=h_!g_*\cM_\bla 
$$
for $l\gg k\gg 0$, where $\bC_{\mathbf{0}}$ is the residue field at the origin in $\bC^r$.

By \cite[R\'esultat 1]{M}, the $\sD_X[\bs]$-module $\cN_{k,l,\bal}$ is relative holonomic over $\bC[\bs]$ and has grade number $n$. We will need in addition the following lemma, analogous to Proposition \ref{propGeN}, and which can be proved similarly:
\begin{lemma}\label{lm:vanishN}  For each $j>n$, if $\Ext^j_{\sD_X[\bs]}(\cN,\sD_X[\bs])\not=0$, the support of $\Ext^j_{\sD_X[\bs]}(\cN,\sD_X[\bs])$ as a $\bC[\bs]$-module is a proper algebraic subset $\subsetneq\bC^r$.
\end{lemma}

\subsection{}
Let $V=X\setminus {D_i}$, and let $g$ and $h$ be defined as in (\ref{eqGH}), so that $j=h\circ g$. Let $$\sD_X(\bs)=\sD_X[\bs]\otimes_{\bC[\bs]} \bC(\bs).$$
For a holonomic $\sD_X(\bs)$-module $\cM$, we have the functor
\[h_!(\cM|_V)=\D\circ h_*\circ \D(\cM|_V)\]
which is also a holonomic $\sD_X(\bs)$-module since the base field $\bC(\bs)$ is of characteristic zero. Similarly for any other open embedding, such as $g$ and $j$. By using the adjoint pairs $(h^{-1}, h_*)$, we have a natural morphism
\[h_!(\cM|_V)\ra h_*(\cM|_V),\]
whose image is $h_{!*}(\cM|_V)$ by definition. We then have natural morphisms 
\[j_!(\cM|_U)\ra h_!g_*(\cM|_U)\ra h_*g_*(\cM|_U)=j_*(\cM|_U).\]

\subsection{}
Now note that $\cN_{k,l,\bal}\otimes_{\bC[\bs]}\bC(\bs)$ is a holonomic $\sD_X(\bs)$-module, so in particular it is $n$-Cohen-Macaulay. Moreover, the morphism 
\[j_!(\cN_{k,l,\bal}\otimes_{\bC[\bs]} \bC(\bs)|_U)\to j_*(\cN_{k,l,\bal}\otimes_{\bC[\bs]} \bC(\bs)|_U)\]
is an isomorphism, both being isomorphic to $\cN_{k',l',\bal}\otimes_{\bC[\bs]} \bC(\bs)$ for every $k', l'$; see the proof of \cite[Theorem 5.3]{WZ}.  The same argument proves that 
\[h_!g_*(\cN_{k,l,\bal}\otimes_{\bC[\bs]} \bC(\bs)|_U)\to h_*g_*(\cN_{k,l,\bal}\otimes_{\bC[\bs]} \bC(\bs)|_U)\]
is an isomorphism, and both are isomorphic to $\cN_{k',l',\bal}\otimes_{\bC[\bs]} \bC(\bs)$. 

We now take $k_0\gg 0$ such that 
$$\cN_{k_0,0,\bal}\otimes_{\bC[\bs]} \bC[\bs]_m=(\sD_X[\bs]\prod_{i=1}^r f_i^{-k_i}\cdot\bf^{\bs-\bk+\bal})\otimes_{\bC[\bs]} \bC[\bs]_m=j_*(\cN|_U\otimes_{\bC[\bs]} \bC[\bs]_m)$$
for all $k\ge k_0$ and all $k_i\ge 0$ with $i=1,\dots r$, where $m$ is the maximal ideal of $\mathbf{0}\in \bC^r$. Such $k_0$ exists, by the existence of multivariate $b$-functions.

Write
\[\bK_{\bal}\coloneqq\bZ^r\setminus \bigcup_{j>n} \supp_{\bC[\bs]}(\Ext^j_{\sD_X[\bs]}(\cN,\sD_X[\bs])).\] 
By Lemma \ref{lm:vanishN}, we can choose some $k\ge k_0$ and some $k_i\ge 0$ with $i=1,\dots,r$ satisfying
\[-(k+k_1,\dots,k+k_r)\in \bK_{\bal}.\]
That is, $-(k+k_1,\dots,k+k_r)$ is not in $\supp_{\bC[\bs]}(\Ext^j_{\sD_X[\bs]}(\cN,\sD_X[\bs]))$ for each $j>n$. Equivalently, by using substitution, for each $j>n$
\[\mathbf{0}\not\in \supp_{\bC[\bs]}(\Ext^j_{\sD_X[\bs]}(\sD_X[\bs]\prod_{i=1}^r f_i^{-k_i}\cdot\bf^{\bs-\bk+\bal},\sD_X[\bs])).\]
 Therefore, $j_*(\cN|_U\otimes \bC[\bs]_m)$ is $n$-Cohen Macaulay over $$\sD_X[\bs]_m=\sD_X[\bs]\otimes_{\bC[\bs]} {\bC[\bs]_m},$$ as $\Ext$ modules localize. In particular, $g_*(\cN|_U\otimes \bC[\bs]_m)$ is $n$-Cohen Macaulay over $\sD_V[\bs]_m$. 
 
\begin{lemma}
\[\bD(\cN|_U)=\bD(\cM_\bla[\bs]\bf^{\bs}) = \bD(\cM_\bla)[\bs]\bf^{-\bs} =\cM_{-\bla}[\bs]\bf^{-\bs}\simeq \cM_{-\bla}[\bs]\bf^{\bs}\]
as $\sD_U[\bs]$-modules, where the third $\bD$ is the $\sD_U$-dual,   and where the last isomorphism is not canonical being given by the substitution $-\bs$ by $\bs$.
\end{lemma}
\begin{proof} Only the second isomorphism needs a proof. We actually prove a slightly more general statement. Let $\bf^\bs \cdot$ be the  functor from the category of coherent left $\sD_U[\bs]$-modules to itself that acts on objects as 
\[ \cM \mapsto \bf^\bs \cdot  \cM := \bf^\bs \sO_U[s] \otimes_{\sO_U[s]}   \cM \]
and acts on a morphism $\varphi:  \cM \to  \cN$ by 
\[ 1 \otimes \varphi: \bf^\bs \sO_U[s] \otimes_{\sO_U[s]}   \cM \to \bf^\bs \sO_U[s] \otimes_{\sO_U[s]}   \cN, \quad \bf^\bs \otimes m \mapsto \bf^\bs \otimes \varphi(m). \]
Then, we claim that
\[ \D( \bf^\bs \cdot  \cM ) = \bf^{-\bs} \cdot \D( \cM).  \]
Then the original statement follows from taking $ \cM = \cM_\bla \otimes_\bC \bC[s]$, and noting that
\[ \D( \cM \otimes_\bC \bC[s]) = \D( \cM) \otimes_\bC \bC[s].\]

Now we prove the claim. Since the statement is local, we may assume $U$ is a coordinate chart, and we abuse notation and let $\sD_U[\bs]$ also denote $ \Gamma(U, \sD_U[\bs])$. Take a free resolution of $ \cM$. 
\[ \ldots \ra \sD_U[\bs]^{m_1} \xto{ d_1}   \sD_U[\bs]^{m_1} \xto{ d_0 }  \sD_U[\bs]^{m_0}  \to  \cM.\]
where the differential $d_i$ is right multiplication by a $m_{i+1} \times m_i$ matrix $P_{i}$ of operators  in $\sD_U[\bs]$ . 
Then, we apply the functor $\bf^\bs \cdot$, and get a resolution
 \[ \ldots \ra \bf^\bs \cdot \sD_U[\bs]^{m_1} \xto{ 1\otimes d_1}  \bf^\bs \cdot  \sD_U[\bs]^{m_1} \xto{  1\otimes d_0 }   \bf^\bs \sD_U[\bs]^{m_0}  \to \bf^\bs \cdot  \cM.\]

For any differential operator $P(s) \in \sD_U[\bs]$, we let $\bf^\bs \cdot P(s) \cdot \bf^{-\bs}$ denote the conjugation by invertible function $\bf^\bs$. Then we have the following isomorphism
 \[\begin{tikzcd}
 \bf^\bs \cdot \sD_U[\bs] \arrow[rr, "(-) \cdot P(s)"] \arrow[d, "\psi"] && \bf^\bs \cdot \sD_U[\bs] \arrow[d, "\psi" ] \\
 \sD_U[\bs] \arrow[rr, "(-) \cdot \bf^\bs \cdot P(s) \cdot \bf^{-\bs}" ] & & \sD_U[\bs]
 \end{tikzcd}
 \]
where $\psi$ is an isomorphism of left $\sD_U[\bs]$-modules, and it sends $ \bf^\bs \otimes Q(s)$ to $ \bf^\bs Q(s) \bf^{-\bs}$. 

Applying the isomorphism $\psi$  to the resolution, we have 
\[ \ldots \ra \sD_U[\bs]^{m_1} \xto{ (-) \cdot \bf^\bs \cdot P_0(s) \cdot \bf^{-\bs} }   \sD_U[\bs]^{m_0}   \to \bf^\bs  \cdot\cM.\]

Then we can take the dual, and get a quasi-isomorphism
\[ \D(\bf^\bs \cdot \cM) \simeq [0 \to \sD_U[\bs]^{m_0} \xto{(-) [\bf^\bs \cdot P_0(s) \cdot \bf^{-\bs}]^*}  \sD_U[\bs]^{m_1}  \to \ldots]  \]
where $[-]^*$ is the formal adjoint in the coordinate chart $U$, and $0$ is sitting at degree $-n$. 

Now, we consider $\bf^{-\bs} \cdot \D(  \cM)$, and we get 
\[ \bf^{-\bs} \cdot \D(  \cM) \simeq [0 \to \sD_U[\bs]^{m_0} \xto{(-) \cdot \bf^{-\bs} P_0(s)^* \bf^{\bs} }  \sD_U[\bs]^{m_1}  \to \ldots. ]\]
It remains to observe that 
\[ [ \bf^\bs \cdot P(s) \cdot \bf^{-\bs} ]^* = \bf^{-\bs} \cdot P(s)^* \cdot  \bf^\bs \]
to see that the two chain complexes are identical. This proves the claim. 
\end{proof}


Since the duality functor commutes with localization, and the substitution $-\bs\mapsto \bs$ takes $\mathbf 0$ to $\mathbf 0$, we further have 
\[\bD(\cN|_U\otimes \bC[\bs]_m)\simeq \cM_{-\bla}[\bs]\bf^{\bs}\otimes \bC[\bs]_m\]
as $\sD_U[\bs]_m$-modules. Then we apply \cite[Theorem 5.3(ii)]{WZ} and conclude that 
$$h_*\bD(g_*(\cN|_U\otimes \bC[\bs]_m))=h_*g_!(\bD(\cN|_U\otimes \bC[\bs]_m))\simeq h_*(\sD_V[\bs]\bf^{\bs+\bk'-\bal})\otimes \bC[\bs]_m$$
and 
$$\bD(j_*(\cN|_U\otimes \bC[\bs]_m))=j_!(\bD(\cN|_U\otimes \bC[\bs]_m))\simeq(\sD_X[\bs]\bf^{\bs+\bk'-\bal})\otimes \bC[\bs]_m$$
for some $k'\gg k_0$. Using the existence of multivariate $b$-functions annihilating the quotient $$\sD_X[\bs]\bf^{\bs+\bk'-\bal}/\sD_X[\bs]\bf^{\bs+\bk'-\bal+\bbe_i},$$ we further know that
\[h_*(\sD_V[\bs]\bf^{\bs+\bk'-\bal})\otimes \bC[\bs]_m=\sD_X[\bs]\bf^{\bs+\bk'-\bal-l\bbe_i}\otimes \bC[\bs]_m\]
 for all $l\gg k'$. Moreover, we can assume
that 
\[\bk'-l\bbe_i\in \bK_{-\bal},\]
by Lemma \ref{lm:vanishN}. That is, $h_*\bD(g_*(\cN|_U\otimes \bC[\bs]_m))$ is $n$-Cohen Macaulay.

We hence conclude by taking the $\sD_X[\bs]$-dual that the complex $h_!g_*(\cN|_U\otimes \bC[\bs]_m)$ is a $n$-Cohen Macaulay module.  In particular, it is also $n$-pure over $\sD_X[\bs]_m$, see for example \cite[3.3]{BVWZ}. We then can define 
\[h_{!*}g_*(\cN|_U\otimes \bC[\bs]_m)\]
to be the image of the natural morphism 
\[h_{!}g_*(\cN|_U\otimes \bC[\bs]_m)\to h_{*}g_*(\cN|_U\otimes \bC[\bs]_m).\]
Then by duality, $h_{!*}g(\cN|_U\otimes \bC[\bs]_m)$ is the minimal extension of $g_*(\cN|_U\otimes \bC[\bs]_m)$.


\subsection{} We next prove that the natural morphism 
\[\eta\colon h_!g_*(\cN|_U\otimes \bC[\bs]_m)\ra h_!g_*(\cN|_U\otimes \bC(\bs))\]
is injective. 
 It is enough to prove that for every $b\in m$, the morphism
 \[h_!g_*(\cN|_U\otimes \bC[\bs]_m)\xrightarrow{\cdot b} h_!g_*(\cN|_U\otimes \bC[\bs]_m)\]
is injective. But, as in the proof of \cite[Lemma 3.4.2]{BVWZ},  the kernel of this morphism for every $b\in m$ must be $0$, because of purity. 
Hence $\eta$ is injective. 

\subsection{}
Now we look at the commutative diagram 
\[
\begin{tikzcd}
 h_!g_*(\cN|_U\otimes \bC[\bs]_m)\arrow[r,"\eta"]\arrow[d] & h_!g_*(\cN|_U\otimes \bC(\bs))\arrow[d]\\
 h_*g_*(\cN|_U\otimes \bC[\bs]_m) \arrow[r] & j_*(\cN|_U\otimes \bC(\bs)).
\end{tikzcd}
\]
The second horizontal morphism is injective by definition and the second vertical morphism is identity as both modules are equal to $\cN_{k,l,\bal}\otimes \bC(\bs)$ for every $k$ and $l$. Since $\eta$ is also injective, the natural morphism 
\[ h_!g_*(\cN|_U\otimes \bC[\bs]_m)\ra h_*g_*(\cN|_U\otimes \bC[\bs]_m)\]
is also injective and hence 
\[h_{!*}g(\cN|_U\otimes \bC[\bs]_m)=h_{!}g(\cN|_U\otimes \bC[\bs]_m)\] 
Since 
$$h_*g_*(\cN|_U\otimes \bC[\bs]_m)=j_*(\cN|_U\otimes \bC[\bs]_m)=\cN_{k,0,\bal}\otimes \bC[\bs]_m$$
for $k>k_0$, by minimality 
\[h_{!*}g(\cN|_U\otimes \bC[\bs]_m)=h_{!}g(\cN|_U\otimes \bC[\bs]_m)=\cN_{k,l,\bal}\otimes \bC[\bs]_m\] 
for all $l\gg k> k_0$.


Since $\bC_\mathbf{0}\simeq \bC[\bs]/m$ is supported at $\mathbf{0}$ in $\bC^r$, we have 
\[\cN_{k,l,\bal}\otimes^L_{\bC[\bs]} \bC_\mathbf{0}\simeq \cN_{k,l,\bal}\otimes \bC[\bs]_m\otimes^L_{\bC[\bs]_m} \bC_\mathbf{0}\simeq h_!g_*(\cN_{k,l,\bal}\otimes \bC[\bs]_m|_U)\otimes^L_{\bC[\bs]_m} \bC_\mathbf{0}\]
for $l\gg k\gg 0$. One can easily check $\D$ and $\bullet\otimes^L_{\bC[\bs]_m} \bC_\mathbf{0}$ commute, and hence the required statement follows by substitution.
$\hfill\Box$

\end{document}